\documentclass[10pt]{paper}%
\usepackage[T1]{fontenc}
\usepackage{amsmath,amssymb}
\usepackage{fullpage}
\usepackage{amsthm}
\usepackage{multirow}
\usepackage{color}
\usepackage{pifont}
\usepackage{amsmath}
\usepackage{graphicx}
  \usepackage[all]{xy}
\usepackage{amsfonts}
\usepackage{amssymb}
\usepackage{ytableau}%

\providecommand{\U}[1]{\protect\rule{.1in}{.1in}}
\providecommand{\U}[1]{\protect\rule{.1in}{.1in}}
\providecommand{\U}[1]{\protect\rule{.1in}{.1in}}
\providecommand{\U}[1]{\protect\rule{.1in}{.1in}}
\providecommand{\U}[1]{\protect\rule{.1in}{.1in}}
\newcommand{\ulambda}{{\boldsymbol{\lambda}}}

\newcommand{\umu}{{\boldsymbol{\mu}}}

\newcommand{\uemptyset }{{\boldsymbol{\emptyset}}}
\newtheorem{Th}{Theorem}[subsection]
\newtheorem{lemma}[Th]{Lemma}

\newtheorem{Prop}[Th]{Proposition}

\theoremstyle{remark}
{\rmfamily}
\theoremstyle{definition}
{\rmfamily}
\newtheorem{exa}[Th]{Example}{\rmfamily}
\newtheorem{abs}[Th]{\bfseries}

\newcommand\blfootnote[1]{%
  \begingroup
  \renewcommand\thefootnote{}\footnote{#1}%
  \addtocounter{footnote}{-1}%
  \endgroup
}
\begin{document}

\title{Two maps on affine type $A$ crystals and Hecke algebras}
\author{N.Jacon}
\maketitle
\date{}
\blfootnote{\textup{2010} \textit{Mathematics Subject Classification}: \textup{20C08,05E10,17B37}} 
\begin{abstract}
We use the crystal isomorphisms of the Fock space to describe two maps 
 on partitions and multipartitions which naturally appear in the crystal basis theory for quantum groups in affine type $A$ and in the representation 
  theory of Hecke algebras of type $G(l,l,n)$. 

\end{abstract}

\section{Introduction}

Since the works of Lascoux, Leclerc, Thibon and Ariki in the $90$'s, it is known  that the representation theory of 
  Hecke algebras of complex reflection groups is closely related to the crystal basis theory for quantum groups. 
   In particular, the crystal basis for Fock spaces in affine type $A$ leads to  a classification of the simple modules of the  Hecke algebra of type $G(l,1,n)$ (also known as Ariki-Koike algebra) in the modular case
     by certain combinatorial objects called Uglov $l$-partitions.  This includes the
     cases of Iwahori-Hecke algebras of type $A$ and $B$ (see \cite{A,GJ}).

A lot of informations on the  representation theory of Hecke algebras of type $D_n$ or, more generally, of type $G(l,l,n)$ 
  can be obtained from  the  $G(l,1,n)$ case. In fact, these latter algebras can be seen as subalgebras of Hecke algebras of type $G(l,1,n)$ and it is possible to 
   produce all the simple modules by studying the restriction  of the simple modules  of the Hecke algebras of type  $G(l,1,n)$,   using Clifford theory. This problem has been studied in various papers using different approaches (see \cite{Hugp,GwJ} and the references theirin). 
 The one  developed in \cite{Hugp} and \cite{Hl}  in particular involves the existence of two maps  which are defined using the crystal graph of an irreducible highest weight module in affine type $A$:
 \begin{itemize}
 \item The first map associates to each Uglov $l$-partition labelling a vertex of the crystal, another Uglov $l$-partition. 
 \item The second one associates to each $e$-regular partition, labelling a vertex of the crystal of the fundamental representation   a certain Uglov $l$-partition. 
 \end{itemize}
 The existence of such maps is non trivial  and based on the structure of the associated crystals and  their  descriptions  are    only recursive on the size of the partitions/multipartitions involved. 
 The interests of these maps is that they allow to describe the restriction of the simple modules of the Hecke algebras of type $G(l,1,n)$ to the Hecke  algebras of type $G(l,l,n)$. 
 
  The aim of this note is to recover, generalize and explicit these results. The main tools of the proof are the  crystal isomorphisms defined and described in \cite{JL}. Using them, the proofs become  purely combinatorial and quite elementary. They also permits to explain how the approaches developed in \cite{Hugp} and in \cite{GwJ} are related.

\section{Crystals}

In this part, we quickly recall some basic combinatorial notions, then we focus on the definition and on important properties of the crystals for Fock spaces. 
 In all this section, we set  $l\in \mathbb{Z}_{>0}$ and  $e\in \mathbb{Z}_{>1}$.  

\subsection{Generalities on Fock spaces and crystals}

\begin{abs}
A {\it partition} is by definition   a nonincreasing
sequence $\lambda=(\lambda_{1},\ldots, \lambda_{m})$ of nonnegative
integers.  If  $\sum_{1\leq i\leq m} \lambda_i=n$, we say that $\lambda$ is a partition of $n$.  
 For $j=1,\ldots, l$, let $\lambda^j$ be a partition of $n_j\in \mathbb{Z}_{\geq 0}$ then 
we say that the $l$-tuple $(\lambda^1,\ldots,\lambda^l)$ is an {\it $l$-partition} 
 of  $n$ if  
  $\sum_{1\leq j\leq l} n_j=n$. We denote by 
   $\Pi^l (n)$  the set of $l$-partitions of rank $n$. 
  The empty $l$-partition is by definition the unique 
  partition of $0$ and it is denoted by $\uemptyset:=(\emptyset,\ldots,\emptyset)$. When $l=1$, the $1$-partitions are identified with the partitions in an obvious way. 
 \end{abs}

 \begin{abs}
  Let  ${\bf s}=(s_1,\ldots,s_l)\in \mathbb{Z}^l$ (we say that ${\bf s}$ is a {\it multicharge}). 
  Let $q$ be an indeterminate. The  $\mathbb{Q} (q)$-vector space generated by all the $l$-partitions:
$$\mathcal{F}_q:=\bigoplus_{n\in \mathbb{Z}_{\geq 0}} \bigoplus_{\ulambda \in \Pi^l (n)} \mathbb{Q} (q) \ulambda$$
is called the {\it Fock space}.  
Let ${\mathcal{U}_{q}(\widehat{\mathfrak{sl}_{e}})}$ be the quantum group of
affine type $A_{e-1}^{(1)}$. This is an associative $\mathbb{Q}(q)$-algebra
with generators $e_i,f_i,t_i,t_i^{-1}$ (for $i=0,\ldots,e-1$) and $\partial$
and relations given in \cite[\S 6.1]{GJ}. We denote by ${\mathcal{U}_q
^{\prime}(\widehat{\mathfrak{sl}_e}) }$ the subalgebra generated by $%
e_i,f_i,t_i,t_i^{-1}$ (for $i=0,...,e-1$).  For $i=0,\ldots,e-1$,   we denote by $\Lambda_i$ the fundamental weights and the simple roots are given by:
  $$\alpha_i= -\Lambda_{i-1}+2\Lambda_i-\Lambda_{i+1},$$ where the indices are taken modulo $e$.   
There is an action of $\mathcal{U}_q^{\prime} (\widehat{\mathfrak{sl}}_e)$  on the Fock space. This action depends on the choice of ${\bf s}$ and the 
 module generated by the empty multipartition  is an irreducible highest weight module with weight $\Lambda_{s_1}+\ldots +\Lambda_{s_l}$.  
 We do not need the precise definition of this action and we refer to \cite[Ch. 6]{GJ} for details.

 \end{abs}

\begin{abs}
     To each $\ulambda\in \Pi^l (n)$ is associated  its {\it   Young tableau}:
$$[\ulambda]=\{(a,b,c)\ |  \ a\geq 1,\ c\in \{0,\ldots,l-1\}, 1\leq b\leq \lambda_a^c\}.$$
We define the {\it content} of a node  $\gamma=(a,b,c)\in [\ulambda]$ as follows: 
$$\text{cont}(\gamma)=b-a+s_c,$$
 and the residue $\mathrm{res}(\gamma )$ is by definition the content of the node taken modulo $e$. 
 We will say that $\gamma $ is an $i+e\mathbb{Z}$-node of ${\boldsymbol{\lambda}}$ when $%
\mathrm{res}(\gamma )\equiv i +e\mathbb{Z}$ (we will sometimes simply called it an $i$-node). Finally, We say that $\gamma $
is {\it removable} when $\gamma =(a,b,c)\in [{\boldsymbol{\lambda}}]$ and $[{%
\boldsymbol{\lambda}}]\backslash \{\gamma \}$ is the Young diagram of  an $l$-partition. Similarly, $%
\gamma $ is {\it addable} when $\gamma =(a,b,c)\notin [{\boldsymbol{\lambda}}]$ and $%
[{\boldsymbol{\lambda}}]\cup \{\gamma \}$ is the Young diagram of  an $l$-partition.

 Let $\gamma$, $\gamma'$ be two removable or addable $i$-nodes of  $\ulambda$.  We  
denote
$$\gamma\prec_{\bf s}\gamma' \qquad 
\stackrel{\text{def}}{\Longleftrightarrow}\qquad \left\{\begin{array}{ll}
\mbox{either} & b-a+s_c<b'-a'+s_{c'},\\
\mbox{or} & b-a+s_c=b'-a'+s_{c'} \text{ and }c>c'.\end{array}\right.$$
\end{abs}
\begin{abs}

For $\ulambda$ an $l$-partition and $i\in \mathbb{Z}/e\mathbb{Z}$, we can consider its set of addable and removable $i$-nodes.  Let $w^{(e,{\bf s})}_{i}(\ulambda)$ be the word obtained first by writing the
addable and removable $i$-nodes of ${\boldsymbol{\lambda}}$ in {increasing}
order with respect to $\prec _{{\mathbf{s}}}$,  
next by encoding each addable $i$-node by the letter $A$ and each removable $%
i$-node by the letter $R$.\ Write $\widetilde{w}^{(e,{\bf s}}_{i}(\ulambda)=A^{p}R^{q}$ for the
word derived from $w^{(e,{\bf s})}_{i}(\ulambda)$ by deleting as many of the factors $RA$ as
possible. In the following, we will sometimes write 
   $\widetilde{w}_{i}(\ulambda)$ and $w_i (\ulambda)$ instead of    $\widetilde{w}^{(e,{\bf s})}_{i}(\ulambda)$ and $w^{(e,{\bf s})}_{i}(\ulambda)$ if there is no possible confusion.

If $p>0,$ let $\gamma $ be the rightmost addable $i$-node in $%
\widetilde{w}_{i}$. The node $%
\gamma $ is called the {good addable $i$-node}. If $r>0$, the leftmost removable $i$-node in 
 $\widetilde{w}_{i}$ is called the {good removable $i$-node}.
 
 \end{abs}
 
\begin{exa}
For $l=2$, ${\bf s}=(0,1)$ and $e=3$. Let us consider the   $2$-partition $\ulambda:=((4),(2,1))$ of $7$. We write its Young tableau and 
 the residues of the nodes in the associated boxes:
  $$
\left(
\begin{array}{|c|c|c|c|}
  \hline
  0& 1  &2 &0   \\
  \hline
\end{array}\;,\;
\begin{array}{|c|c|}
  \hline
  1  &2      \\
  \cline{1-2}
 0   \\
 \cline{1-1}
\end{array}
\right)$$
We have $\widetilde{w}_{0}(\ulambda)=RAR$ and thus $(1,4,1)$ is a good removable $0$-node for $\ulambda$. We have 
$\widetilde{w}_{2}(\ulambda)=AAR$ and thus $(2,1,1)$ is a good addable $2$-node for $\ulambda$ and $(1,2,2)$ 
 is a good removable $2$-node for it.

\end{exa}

\begin{abs}

We denote by 
$\mathcal{G}_{e,{\bf s}}$ the crystal of the Fock space computed using the Kashiwara operators  $\widetilde{e}^{e,{\bf s}}_i$ and  $\widetilde{f}^{e,{\bf s}}_i$. Again, we refer to \cite{GJ} for details. This is the graph with 
\begin{itemize}
\item vertices : the $l$-partitions $\ulambda\vdash_l n$ with $n\in \mathbb{Z}_{\geq 0}$.
\item  arrows: $\ulambda \overset{i}{\rightarrow }\umu$  if and only if     $\widetilde{f}^{e,{\bf s}}_i  \ulambda=\umu$ (or equivalently $\widetilde{e}^{e,{\bf s}}_i  \umu=\ulambda$). This means that ${{\umu}}$ is obtained by
adding to ${{\ulambda}}$ a good addable $i$-node, or equivalently, $\ulambda$ is obtained from $\umu$ 
 by removing a good removable $i$-node. 
\end{itemize}

\end{abs}
\begin{exa}\label{gr}
For $l=3$,  $e=2$ and ${\bf s}=(0,0,1)$ the graph below is the subgraph of $\mathcal{G}_{e,{\bf s}}$ containing the empty $3$-partition and with the $3$-partitions with rank less or equal than $4$.

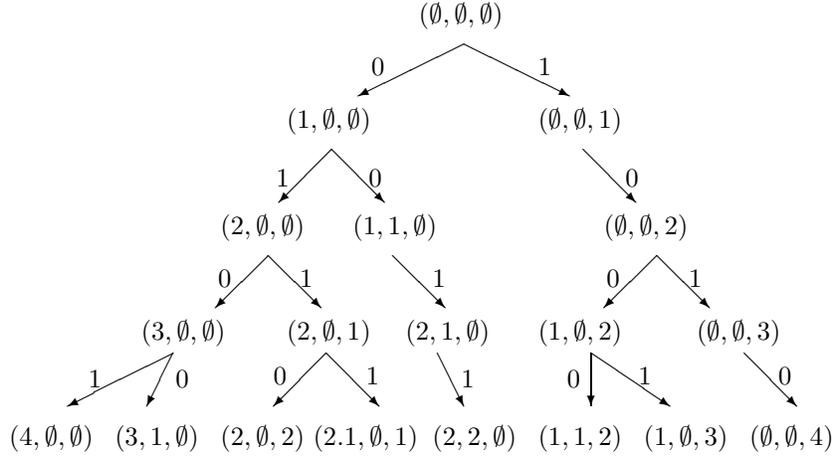
\begin{center}
\begin{picture}(250,150)
\put(  125, 140){$(\emptyset,\emptyset,\emptyset)$}

 \put( 75, 100){$(1,\emptyset,\emptyset)$}
  \put( 170, 100){$(\emptyset,\emptyset,1)$}

\put( 50, 60){$(2,\emptyset,\emptyset)$}
\put( 100, 60){$(1,1,\emptyset)$} 
\put( 195, 60){$(\emptyset,\emptyset,2)$}

\put( 20, 20){$(3,\emptyset,\emptyset)$}
\put( 75, 20){$(2,\emptyset,1)$} 
\put( 120, 20){$(2,1,\emptyset)$} 
\put( 170, 20){$(1,\emptyset,2)$} 
\put( 230, 20){$(\emptyset,\emptyset,3)$}

\put( -30, -20){$(4,\emptyset,\emptyset)$}
\put( 10, -20){$(3,1,\emptyset)$} 
\put( 50, -20){$(2,\emptyset,2)$} 
\put( 85,- 20){$(2.1,\emptyset,1)$} 
\put( 130, -20){$(2,2,\emptyset)$} 
\put( 170,-20){$(1,1,2)$}
\put( 210, -20){$(1,\emptyset,3)$} 
\put( 250,-20){$(\emptyset,\emptyset,4)$}

 \put( 142, 132){\vector(-2,-1){40}} 
  \put(142, 132){\vector(2,-1){40}}

 \put( 92, 92){\vector(-1,-1){20}} \put( 92, 92){\vector(1,-1){20}} 
   \put(187, 92){\vector(1,-1){20}}

 \put( 68, 52){\vector(-1,-1){20}} \put( 68, 52){\vector(1,-1){20}} 
   \put(115, 52){\vector(1,-1){20}} 
 \put( 215, 52){\vector(-1,-1){20}} \put( 215, 52){\vector(1,-1){20}}

 \put( 32, 15){\vector(-2,-1){40}} \put( 32, 15){\vector(-1,-2){10}}

  \put( 132, 15){\vector(1,-2){10}} 
   \put(90, 15){\vector(-1,-1){20}}   \put(90, 15){\vector(1,-1){20}} 
 \put( 190, 15){\vector(0,-1){20}} \put( 190, 15){\vector(3,-2){30}} 
\put( 248, 15){\vector(1,-1){20}}

\put(  107, 120){$0$} 
\put(  170, 120){$1$}

\put(  203, 78){$0$} 
\put(  106, 78){$0$} 
 \put(  71, 78){$1$}

\put(  196, 40){$0$} 
 \put(  227, 40){$1$}  
\put(  80, 40){$1$} 
\put(  49, 40){$0$} 
\put(  130, 40){$1$}

\put(  261, 3){$0$} 
\put(  181, 2){$0$} 
\put(  141, 2){$1$} 
\put(  208, 3){$1$} 
 \put(  0, 2){$1$} 
\put(  33, 2){$0$} 
\put(  70, 3){$0$} 
 \put(  105, 3){$1$} 
\end{picture} 
\end{center}

\end{exa}

\vspace{1cm} 

\begin{abs}\label{flotw}

Let $\Phi_{e, {\bf s}} (n)$ to be the set of $l$-partitions of rank $n$ in the connected component of $\mathcal{G}_{e,{\bf s}}$  containing the empty $l$-partition. This is called the set of {\it  Uglov $l$-partitions}.  Hence, by definition, an Uglov $l$-partition is defined  by adding successively good nodes to the empty $l$-partition (with arbitrary residues).  
It strongly  depends on the choice of ${\bf s}$. Assume that ${\bf s}$ is such that  $0<s_j-s_{i} <e$ for all $0<i<j\leq l$ then the set $\Phi_{e, {\bf s}} (n)$ is known as the set of FLOTW $l$-partitions and it has a nice non recursive description (see \cite[\S 6.3.2]{GJ}).   We have $\ulambda=(\lambda^1,\ldots,\lambda^l)\in \Phi_{{\bf s},e} (n)$ 
 if and only if:
 \begin{enumerate}
 \item For all $j=1,\ldots,l-1$ and $i\in \mathbb{Z}_{>0}$, we have:
 $$\lambda_i^j\geq \lambda_{i+s_{j+1}-s_j}^{j+1}.$$
 \item For all $i\in \mathbb{Z}_{>0}$, we have:
 $$\lambda_i^{l}\geq \lambda_{i+e+s_{1}-s_l}^{1}.$$ 
 \item For all $k\in \mathbb{Z}_{>0}$, the set 
 $$\{ \lambda_i^j-i+s_j+e\mathbb{Z}\ |\ i\in \mathbb{Z}_{>0},\ \lambda_i^j=k, j=1,\ldots,l\},$$
 is a proper subset of $\mathbb{Z}/e\mathbb{Z}$. 
 \end{enumerate}
In general, we don't have such a nice description of the set of Uglov $l$-partitions.

\end{abs}
\begin{exa}
In the case where $l=1$, the set $\Phi_{e, (0)} (n)$ is the set of $e$-regular partitions of $n$, that is, the set of partitions of rank $n$ such that no non zero parts are repeated $e$ or more times.  

\end{exa}

\begin{exa}
Following Example \ref{gr}, we have 
$$\Phi_{2, (0,0,1)} (4)=\{(4,\emptyset,\emptyset),(3,1,\emptyset),(2,\emptyset,2), (2.1,\emptyset,1), (2,2,\emptyset), (1,1,2), (1,\emptyset,3), (\emptyset,\emptyset,4)\}$$

\end{exa}
\subsection{Crystal isomorphisms}

In this part, we recall the definition of certain crystal isomorphisms  studied in \cite{JL}. These maps will be intensively used in the next sections.

   \begin{abs}\label{aff}
Let $\widehat{\mathfrak{S}}_l$ be the (extended) affine symmetric group. This is defined as follows. We denote by $P_l:=\mathbb{Z}^l$ the $\mathbb{Z}$-module with  standard basis $\{y_i\ |\ i=1,\ldots, l\}$. For $i=1,\ldots,l-1$, we denote by $\sigma_i$ the transposition $(i,i+1)$ of  $\mathfrak{S}_l$. Then 
$\widehat{\mathfrak{S}}_l$ can be seen as 
  the semi-direct product $P_l \rtimes \mathfrak{S}_l$ where the relations  are given by $\sigma_i y_j=y_j \sigma_i$ for $j\neq i,i+1$ and $\sigma_i y_i \sigma_i=y_{i+1}$ for $i=1,\ldots,l-1$ and $j=1,\ldots,l$.  
This group acts  on $\mathbb{Z}^l$  by setting for any ${{{\bf s}}}=(s_{1},\ldots ,s_{l})\in 
\mathbb{Z}^{l}$: %
$$\begin{array}{rcll}
\sigma _{c}.{{{\bf s}}}&=&(s_{1},\ldots ,s_{c-1},s_{c+1},s_{c},s_{c+2},\ldots ,s_{l})&\text{for }c=1,\ldots,l-1 \text{ and }\\
y_i.{{{\bf s}}}&=&(s_{1},s_{2},\ldots,s_i+e,\ldots ,s_{l})&\text{for }i=1,\ldots,l.
\end{array}$$
A fundamental domain for this action is given by
$$\mathcal{A}^e_l:=\left\{(s_1,\ldots  ,s_{l})\in \mathbb{Z}^l\ |\ 0\leq s_1 \leq \ldots  \leq s_l <e \right\}.$$
Note that we thus have a description of   $\Phi_{{\bf s},e} (n)$ when ${\bf s}$ is in this domain by \S \ref{flotw}. 
Let  
 $\tau:=y_l \sigma_{l-1}\ldots\sigma_1$ then  we see that 
    $\widehat{\mathfrak{S}}_l$ is generated by $\tau$ and $\sigma_i$ for $i=1,\ldots,l-1$. In addition, we have:
   $$\tau.{\bf s}=(s_2,\ldots,s_{l},s_1+e).$$
  \end{abs}

\begin{abs}\label{crt}
Assume that ${\bf s}\in \mathbb{Z}^l$ and ${\bf s}'\in \mathbb{Z}^l$ are in the same orbit modulo $\widehat{\mathfrak{S}}_l$. 
 As explained in \cite[\S 6.2.17]{GJ}, the crystal graph theory allows to construct a combinatorial bijection between the two sets of Uglov $l$-partitions 
  $\Phi_{{\bf s},e} (n)$ and  $\Phi_{{\bf s}',e} (n)$. Let  $\ulambda \in \Phi_{{\bf s},e} (n)$ then there exists a sequence $(i_1,\ldots,i_n)\in (\mathbb{Z}/e\mathbb{Z})^n$ such that:
 $$\widetilde{f}^{e,{\bf s}}_{i_1} \ldots \widetilde{f}^{e,{\bf s}}_{i_n} \emptyset =\ulambda$$
Then there exists  $\umu \in \Phi_{{\bf s}',e} (n)$ such that
$$\widetilde{f}^{e,{\bf s}'}_{i_1} \ldots \widetilde{f}^{e,{\bf s}'}_{i_n} \emptyset =\umu$$
We set $\Psi^{e}_{{\bf s}\to {\bf s}' }(\ulambda):=\umu$ (it does not depends on the choice of the sequence $(i_1,\ldots,i_n)$). This defines
a bijection 
  $$\Psi^e_{{\bf s}\to {\bf s}'}  :  \Phi_{{\bf s},e} (n) \to  \Phi_{{\bf s}',e} (n).$$
A combinatorial description of this map is given in \cite{JL}. Let us  quickly explain   how. There exists $w\in \widehat{\mathfrak{S}}_l$ such that ${\bf s'}=w.{\bf s}$. Then,
 $w$ is a product of $\tau$ and $\sigma_i$'s ($i=1,\ldots,l-1$). Thus $\Psi^{e}_{{\bf s}\to {\bf s}'} $ is a composition of maps of the form 
 $\Psi^e_{{\bf v}\to \tau.{\bf v}}$ and $\Psi^{e}_{{\bf v}\to \sigma_i.{\bf v}}$  with ${\bf v}\in \mathbb{Z}^l$   explicit by induction. 
 \begin{itemize}
 \item For all $\ulambda=(\lambda^1,\ldots,\lambda^l) \in \Phi_{e, {\bf v}} (n)$, we  have 
 $$\Psi^e_{{\bf v}\to\tau.{\bf v}} (\ulambda)=(\lambda^2,\ldots,\lambda^l,\lambda^1).$$
 \item  For all $\ulambda=(\lambda^1,\ldots,\lambda^l) \in \Phi_{e, {\bf v}} (n)$, we  have 
  $$\Psi^e_{{\bf v}\to\sigma_i .{\bf v}} (\ulambda)=(\lambda^1,\ldots,\lambda^{i-1},\widetilde{\lambda}^{i+1},\widetilde{\lambda}^{i},\lambda^{i+2},\ldots,\lambda^l),$$
 where   $(\widetilde{\lambda}^{i+1},\widetilde{\lambda}^{i})$ is obtained from  $(\lambda^i,\lambda^{i+1})$ via a purely simple combinatorial process described in \cite[th. 5.4.2]{JL} (in terms of Lusztig symbols) 
  or in \cite[\S 5.3]{J} (in terms of Young tableaux).

 \end{itemize}
 
\end{abs}
\begin{abs}
Assume that ${\bf s}=(s_1,\ldots,s_l)$ satisfies   $s_i-s_j \geq n-1-e$ for all $i<j$ then we say that ${\bf s}$ is {\it very dominant}. If both ${\bf s}$ and 
 ${\bf s}'$ are very dominant (comparing to $n$) and in the same orbit then $\Psi^e_{{\bf s}\to {\bf s}'}$ is the identity and the set 
 $\Phi_{{\bf s},e} (n)$ is   known as the set of Kleshchev $l$-partitions (see \cite[Ex. 6.2.16]{GJ}). 

If ${\bf s}\in \mathbb{Z}^l$, one way to compute the set  $\Phi_{{\bf s},e} (n)$ of Uglov $l$-partitions consists in     finding $w\in \widehat{\mathfrak{S}}_l$
 such that ${\bf s}'=w.{\bf s}\in \mathcal{A}^e_l$. We can then use the description
  of the set $\Phi_{{\bf s}',e} (n)$ in \S \ref{flotw} 
  and then apply the isomorphism $\Psi^e_{{\bf s}'\to {\bf s}}$.

\end{abs}

\begin{exa}

For $l=3$,  $e=2$ and ${\bf s}=(2,0,3)$ the graph below is the subgraph of $\mathcal{G}_{e,{\bf s}}$ containing the empty $3$-partition.

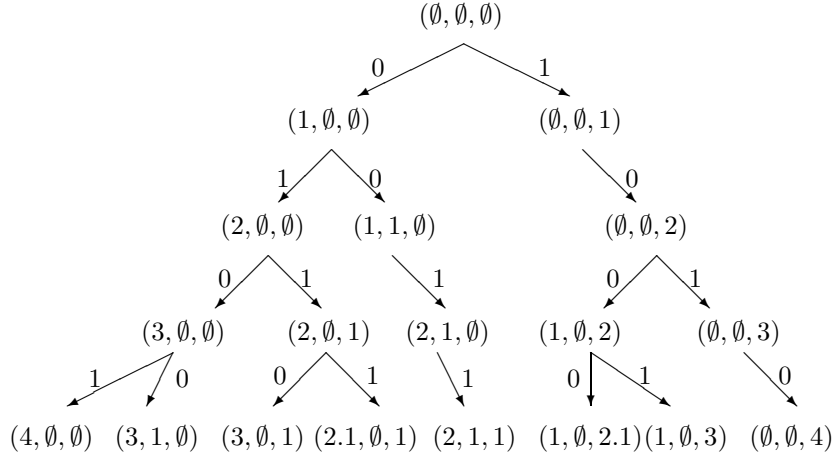
\begin{center}
\begin{picture}(250,150)
\put(  125, 140){$(\emptyset,\emptyset,\emptyset)$}

 \put( 75, 100){$(1,\emptyset,\emptyset)$}
  \put( 170, 100){$(\emptyset,\emptyset,1)$}

\put( 50, 60){$(2,\emptyset,\emptyset)$}
\put( 100, 60){$(1,1,\emptyset)$} 
\put( 195, 60){$(\emptyset,\emptyset,2)$}

\put( 20, 20){$(3,\emptyset,\emptyset)$}
\put( 75, 20){$(2,\emptyset,1)$} 
\put( 120, 20){$(2,1,\emptyset)$} 
\put( 170, 20){$(1,\emptyset,2)$} 
\put( 230, 20){$(\emptyset,\emptyset,3)$}

\put( -30, -20){$(4,\emptyset,\emptyset)$}
\put( 10, -20){$(3,1,\emptyset)$} 
\put( 50, -20){$(3,\emptyset,1)$} 
\put( 85,- 20){$(2.1,\emptyset,1)$} 
\put( 130, -20){$(2,1,1)$} 
\put( 170,-20){$(1,\emptyset,2.1)$}
\put( 210, -20){$(1,\emptyset,3)$} 
\put( 250,-20){$(\emptyset,\emptyset,4)$}

 \put( 142, 132){\vector(-2,-1){40}} 
  \put(142, 132){\vector(2,-1){40}}

 \put( 92, 92){\vector(-1,-1){20}} \put( 92, 92){\vector(1,-1){20}} 
   \put(187, 92){\vector(1,-1){20}}

 \put( 68, 52){\vector(-1,-1){20}} \put( 68, 52){\vector(1,-1){20}} 
   \put(115, 52){\vector(1,-1){20}} 
 \put( 215, 52){\vector(-1,-1){20}} \put( 215, 52){\vector(1,-1){20}}

 \put( 32, 15){\vector(-2,-1){40}} \put( 32, 15){\vector(-1,-2){10}}

  \put( 132, 15){\vector(1,-2){10}} 
   \put(90, 15){\vector(-1,-1){20}}   \put(90, 15){\vector(1,-1){20}} 
 \put( 190, 15){\vector(0,-1){20}} \put( 190, 15){\vector(3,-2){30}} 
\put( 248, 15){\vector(1,-1){20}}

\put(  107, 120){$0$} 
\put(  170, 120){$1$}

\put(  203, 78){$0$} 
\put(  106, 78){$0$} 
 \put(  71, 78){$1$}

\put(  196, 40){$0$} 
 \put(  227, 40){$1$}  
\put(  80, 40){$1$} 
\put(  49, 40){$0$} 
\put(  130, 40){$1$}

\put(  261, 3){$0$} 
\put(  181, 2){$0$} 
\put(  141, 2){$1$} 
\put(  208, 3){$1$} 
 \put(  0, 2){$1$} 
\put(  33, 2){$0$} 
\put(  70, 3){$0$} 
 \put(  105, 3){$1$} 
\end{picture} 
\end{center}

\vspace{1cm}

Looking at example \ref{gr}, we see that  $\Psi^2_{(0,0,1)\to (2,0,3)}$ is the identity for the Uglov $3$-partitions of ranks $\leq 3$, and we have 
$\Psi^2_{(0,0,1)\to (2,0,3)} (2,\emptyset,2)=(3,\emptyset,1)$, 
$\Psi^2_{(0,0,1)\to (2,0,3)} (2,2,\emptyset)=(2,1,1)$, 
$\Psi^2_{(0,0,1)\to (2,0,3)} (1,1,2)=(1,\emptyset,2.1)$, and  $\Psi^3_{(0,0,1)\to (2,0,3)}$ is the identity for the others  Uglov $3$-partitions of ranks $4$. 
This formulae can be also obtained without looking at the crystal. Indeed, we have $(2,0,3)=(\sigma_1 \tau)^2 (0,0,1)$  and thus 
$\Psi^2_{(0,0,1)\to (2,0,3)}=\Psi^2_{(0,0,1)\to \tau (0,0,1)}\circ \Psi^2_{\tau (0,0,1)\to \sigma_1 \tau (0,0,1)} \circ  \Psi^3_{\sigma_1 \tau (0,0,1)\to \tau \sigma_1 \tau (0,0,1)}
\circ \Psi^2_{\tau \sigma_1 \tau (0,0,1)\circ \sigma_1\tau \sigma_1 \tau (0,0,1)}$ and one can use the combinatorial description of the isomorphisms.
\end{exa}

\section{Two maps on crystals}

The above results will allow us to recover and give precisions on  two results on crystals in affine type $A$ thanks to quite elementary proofs.  These two results concern a particular choice of multicharge which  naturally appears in the context of Hecke algebras, as we will see in the next part.  In this section, we thus assume that $l$ divides $e$ and that 
 ${\bf s}\in \mathbb{Z}^l$ is in the orbit of  the multicharge $(0,e/l,\ldots,(l-1)e/l)$ modulo $\widehat{\mathfrak{S}}_l$.

\subsection{Hu's map}

The  first result that we want ro recover is in fact a  direct generalization of a result by Hu \cite[Theorem 3.6]{Hugp}. We propose here an elementary proof of this result using our 
 crystal isomorphisms and we will also give a general method to explicitly compute the $l$-partitions that this results allows to define. 

 \begin{Prop}[Hu] \label{hup}
 Assume that $\ulambda \in \Phi_{e, {\bf s}} (n)$. Then there exists a sequence $(i_1,\ldots,i_n)\in (\mathbb{Z}/e\mathbb{Z})^n$ such that:
$$\widetilde{f}^{e,{\bf s}}_{i_1} \ldots \widetilde{f}^{e,{\bf s}}_{i_n} \emptyset =\ulambda.$$
Then for any such sequences,  there exists $\umu \in \Phi_{e, {\bf s}} (n)$ such that 
$$\widetilde{f}^{e,{\bf s}}_{i_1+e/l} \ldots \widetilde{f}^{e,{\bf s}}_{i_n+e/l} \emptyset =\umu$$
 \end{Prop}
 
 To prove this proposition, we will proceed in two steps, we first prove the proposition for  a particular choice of multicharge which is in $\mathcal{A}^e_l$, and 
  then make use of the crystal isomorphisms we have already defined. 
 
  \begin{lemma}\label{Lhup}
  Let ${\bf s }=(0,e/l,\ldots,(l-1)e/l)$
 Assume that $\ulambda \in \Phi_{e, {\bf s}} (n)$ and that we have  a sequence $(i_1,\ldots,i_n)\in (\mathbb{Z}/e\mathbb{Z})^n$ such that:
$$\widetilde{f}^{e,{\bf s}}_{i_1} \ldots \widetilde{f}^{e,{\bf s}}_{i_n} \emptyset =\ulambda.$$
Then we have that  $\umu:=(\lambda^l,\lambda^1,\ldots,\lambda^{l-1})\in \Phi_{e, {\bf s}} (n)$ and we have
$$\widetilde{f}^{e,{\bf s}}_{i_1+e/l} \ldots \widetilde{f}^{e,{\bf s}}_{i_n+e/l} \emptyset =\umu.$$
 \end{lemma}
 
 \begin{proof}
We  argue by induction on $n$.   The lemma is clear for the empty $l$-partition. Assume now that $n>0$. 
 Let $\ulambda \in \Phi_{e, {\bf s}} (n)$ and assume that we have  a sequence $(i_1,\ldots,i_n)\in (\mathbb{Z}/e\mathbb{Z})^n$ such that 
 $$\widetilde{f}^{e,{\bf s}}_{i_1} \ldots \widetilde{f}^{e,{\bf s}}_{i_n} \emptyset =\ulambda.$$
 Set   $\ulambda':= \widetilde{f}^{e,{\bf s}}_{i_2} \ldots \widetilde{f}^{e,{\bf s}}_{i_n} \emptyset$ then $\ulambda'=({\lambda'}^1,\ldots,{\lambda'}^l)$ is in $\Phi_{e, {\bf s}} (n-1)$  and by induction, 
  $\umu':=({\lambda'}^l,{\lambda'}^1,\ldots,{\lambda'}^{l-1})$ is in $\Phi_{e, {\bf s}} (n-1)$ and we have  $\umu'= \widetilde{f}^{e,{\bf s}}_{i_2+e/l} \ldots \widetilde{f}^{e,{\bf s}}_{i_n+e/l} \emptyset$.

  Now, by hypothesis,  we have that $\gamma=[\ulambda]/[\ulambda']$ is a  good addable $i_n$-node for $\ulambda'$ (and $(e,{\bf s})$). Set $(a,b,c):=\gamma$. Then  $\gamma':=[\umu]/[\umu']$ is 
   an addable node for $\umu'$.  We have $\gamma'=(a,b,c+1)$ (where the $3$rd component is understood modulo $l$) and, by our choice of multicharge, it  is a $i_n+e/l$-node. We want to show that this is a  good addable node for $\umu'$.

Assume that $\gamma_1=(a_1,b_1,c_1)$ is a  $i_n$-addable or removable node for $\ulambda'$. Then $\gamma_1 ':=(a_1,b_1,c_1+1)$ 
 is an  addable or removable  node for $\umu'$ (and it is removable, resp. addable, if and only if $\gamma_1$ is).  We have that 
 $$\operatorname{cont} (\gamma_1')= \left\{ \begin{array}{rr}
 \operatorname{cont} (\gamma_1)+e/l & \text{if } c_1\neq l, \\
  \operatorname{cont} (\gamma_1)+e/l-e & \text{if } c_1=l. \end{array}
\right.$$
  Thus, 
 we have $\gamma_1 \prec_{e,{\bf s}} \gamma$ if and only if $\gamma_1' \prec_{e,{\bf s}} \gamma'$. 
 
 Reciprocally, if $\gamma_1'=(a_1,b_1,c_1)$ is a  $i_n+e/l$-addable or removable node for $\umu'$. Then $\gamma_1 :=(a_1,b_1,c_1-1)$ 
 is an  addable or removable for $i_n$-node for $\ulambda'$ (and it is removable, resp. addable, if and only if $\gamma_1'$ is).  We have that 
 $\gamma_1 \prec_{e,{\bf s}} \gamma$ if and only if $\gamma_1' \prec_{e,{\bf s}} \gamma'$. 
  Thus we have $w_{i_1}^{(e,{\bf s})} (\ulambda')=w_{i_1+e/l}^{(e,{\bf s})} (\umu')$. 
 
 This discussion  implies that 
$\gamma'$ is good addable $i_1+e/l$-node for $\umu'$ and thus that 
$$\widetilde{f}^{e,{\bf s}}_{i_1+e/l} \ldots \widetilde{f}^{e,{\bf s}}_{i_n+e/l} \emptyset =\umu,$$
as required.
 \end{proof}
 
 We can now give a proof of Proposition \ref{hup}. Assume that  ${\bf s}\in \mathbb{Z}^l$ is in the orbit of $(0,e/l,\ldots,(l-1)e/l)$ modulo $\widehat{\mathfrak{S}}_l$, that $\ulambda \in \Phi_{e, {\bf s}} (n)$ and that we have  a sequence $(i_1,\ldots,i_n)\in (\mathbb{Z}/e\mathbb{Z})^n$ such that:
$$\widetilde{f}^{e,{\bf s}}_{i_1} \ldots \widetilde{f}^{e,{\bf s}}_{i_n} \emptyset =\ulambda.$$
Set ${\bf s}':=(0,e/l,\ldots,(l-1)e/l)$, then by definition,  we have 
$$\widetilde{f}^{e,{\bf s}'}_{i_1} \ldots \widetilde{f}^{e,{\bf s}'}_{i_n} \emptyset =\Psi^e_{{\bf s}\to{\bf s}'}  (\ulambda).$$
We can thus use Lemma \ref{Lhup} to deduce that there exists $\umu'  \in \Phi_{e, {\bf s}'} (n)$ such that
$$\widetilde{f}^{e,{\bf s}'}_{i_1+e/l} \ldots \widetilde{f}^{e,{\bf s}'}_{i_n+e/l} \emptyset =\umu',$$
and using again our crystal isomorphism, we get that:
$$\widetilde{f}^{e,{\bf s}}_{i_1+e/l} \ldots \widetilde{f}^{e,{\bf s}}_{i_n+e/l} \emptyset =\Psi^{e}_{{\bf s}'\to {\bf s}} (\umu'),$$
so the result follows. Note in addition that the $l$-partition $\umu'$ may be explicitly described  thanks to the explicit description of the crystal isomorphism without the computation of the crystal itself. 
\begin{exa}
Take $l=2$ and assume that $e=4$. We set ${\bf s}=(0,10)$ which is in the orbit  of ${\bf s}'=(0,2)\in \mathcal{A}^4_2$.  We take $\ulambda=(1.1,5.1)\in \Phi^4_{(0,10)}(8)$.
 Note that the multicharge is very dominant for $2$-partitions of rank $8$ so $(1.1,5.1)$ is a Kleshchev bipartition.  If we want to find the bipartition $\umu$ of Proposition \ref{hup}, 
  we first need to find the bipartition $\ulambda'\in  \Phi^4_{(0,2)}(8)$ such that 
$\Psi^4_{{\bf s} \to {\bf s}'} (\ulambda)=\ulambda'$. Using our description of the isomorphisms, we get $\ulambda'=(2.1,5)$. We then have 
 $\umu=\Psi^4_{{\bf s}' \to {\bf s}} (5,2.1)= (4,3.1)$. 
\end{exa}

\subsection{The map $\iota^{{\bf s}}_k$}
We keep the hypothesis that ${\bf s}$ is in the orbit of the multicharge $(0,e/l,\ldots,(l-1)e/l)$ modulo $\widehat{\mathfrak{S}}_l$. 
We now consider another map defined using the crystal.  As explained in \cite{Hl},
 its existence follows from \cite{NS} in the case where $l=2$. Here we will give a general version and  
  we will again give an easy proof using our crystal isomorphism. We will also  
  make things more explicit.
  \begin{Prop}\label{H0} Let $k$ be an integer dividing $l$. Set ${\bf v}:=(0,e/l,\ldots,e(k-1)/l)\in \mathbb{Z}^k.$ 
  There exists a unique map
  $$\iota^{{\bf s}}_k : \Phi_{ke/l, {\bf v}} (n) \to \Phi_{e, {\bf s}} (ln/k)$$
  well-defined as follows. For all $\ulambda \in \Phi_{ke/l, {\bf v}}(n)$, there exists $(i_1,\ldots,i_n)\in \mathbb{Z}^n$ such that 
 $$\widetilde{f}^{ke/l, {\bf v}}_{i_1} \ldots \widetilde{f}^{ke/l, {\bf v}}_{i_n} \emptyset =\ulambda,$$
 (the indices are understood modulo $ke/l$.)
 Then for all such sequences, we have:
 $$\underbrace{\widetilde{f}^{e,{\bf s}}_{{i}_1}\widetilde{f}^{e,{\bf s}}_{{i}_1+ke/l}  \ldots \widetilde{f}^{e,{\bf s}}_{{i}_1+e-ke/l}}_{l/k}  \ldots  \underbrace{\widetilde{f}^{e,{\bf s}}_{{i}_n}\widetilde{f}^{e,{\bf s}}_{{i}_n+ke/l}  \ldots \widetilde{f}^{e,{\bf s}}_{{i}_n+e-ke/l}}_{l/k} \emptyset =\iota^{{\bf s}}_k (\ulambda)$$
  (the indices are understood modulo $e$.)
  \end{Prop}

In the same spirit as the last result, our strategy consists in proving the result when ${\bf s}$ is in the fundamental domain $\mathcal{A}^e_l$.
\begin{lemma}\label{H1}  
Let $k$ be an integer dividing $l$. Set ${\bf s}=(0,e/l,\ldots,(l-1)e/l)$ and ${\bf v}:=(0,e/l,\ldots,e(k-1)/l)\in \mathbb{Z}^k$. 
  There exists a unique map
  $$\iota^{\bf s}_k : \Phi_{ke/l, {\bf v}} (n) \to \Phi_{e, {\bf s}} (ln/k)$$
  well-defined as follows. For all $\ulambda \in \Phi_{ke/l, {\bf v}} (n)$, there exists $(i_1,\ldots,i_n)\in \mathbb{Z}^n$ such that 
 $$\widetilde{f}^{ke/l, {\bf v}}_{i_1} \ldots \widetilde{f}^{ke/l, {\bf v}}_{i_n} \emptyset =\ulambda$$
 Then for all such sequences, we have:
 $$\underbrace{\widetilde{f}^{e,{\bf s}}_{{i}_1}\widetilde{f}^{e,{\bf s}}_{{i}_1+ke/l}  \ldots \widetilde{f}^{e,{\bf s}}_{{i}_1+e-ke/l}}_{l/k}  \ldots  \underbrace{\widetilde{f}^{e,{\bf s}}_{{i}_n}\widetilde{f}^{e,{\bf s}}_{{i}_n+ke/l}  \ldots \widetilde{f}^{e,{\bf s}}_{{i}_n+e-ke/l}}_{l/k} \emptyset =(\underbrace{\lambda^1,\ldots,\lambda^{k}}_{l/k},\underbrace{\lambda^1,\ldots,\lambda^{k}}_{l/k},\ldots, \underbrace{\lambda^1,\ldots,\lambda^{k}}_{l/k}).$$

\end{lemma}
\begin{proof}
We again argue by induction on $n\in \mathbb{Z}_{\geq 0}$.   The lemma is clear for the empty $l$-partition. Assume now that $n>0$. 
 Let $\ulambda \in \Phi_{ke/l, {\bf v}} (n)$  and assume that we have  a sequence $(i_1,\ldots,i_n)\in \mathbb{Z}^n$ such that 
 $$\widetilde{f}^{ke/l, {\bf v}}_{i_1} \ldots \widetilde{f}^{ke/l, {\bf v}}_{i_n} .\emptyset =\ulambda$$
 Set $\ulambda':=\widetilde{f}^{ke/l, {\bf v}}_{i_2} \ldots \widetilde{f}^{ke/l, {\bf v}}_{i_n} .\emptyset $. 
 By induction,  we have:
 
 $$\underbrace{\widetilde{f}^{e,{\bf s}}_{{i}_1}\widetilde{f}^{e,{\bf s}}_{{i}_1+ke/l}  \ldots \widetilde{f}^{e,{\bf s}}_{{i}_1+e-ke/l}}_{l/k}  \ldots  \underbrace{\widetilde{f}^{e,{\bf s}}_{{i}_n}\widetilde{f}^{e,{\bf s}}_{{i}_n+ke/l}  \ldots \widetilde{f}^{e,{\bf s}}_{{i}_n+e-ke/l}}_{l/k}. \emptyset =({\lambda'}^1,\ldots,{\lambda'}^{k},{\lambda'}^1,\ldots,{\lambda'}^{k},\ldots, {\lambda'}^1,\ldots,{\lambda'}^{k})$$
 Denote 
 $$\ulambda[0]:=({\lambda'}^1,\ldots,{\lambda'}^{k},{\lambda'}^1,\ldots,{\lambda'}^{k},\ldots, {\lambda'}^1,\ldots,{\lambda'}^{k})$$

Set $\gamma=[\ulambda]/[\ulambda']$ and let $(a,b,c):=\gamma$. This is a good addable $i_1+(ke/l) \mathbb{Z}$-node for $\ulambda'$. We have by definition $b-a+(c-1)e/l \equiv i_1+(ke/l)\mathbb{Z}$. 
 So there exists $j\in \{0,1,\ldots,l/k-1\}$ such that  $b-a+(c-1)e/l=i_1-j(ke/l)+e\mathbb{Z}$. We thus have $b-a+(c-1+(j-1)k)e/l=i_1-(ke/l)+e\mathbb{Z}$. 
 Let use denote $\gamma_j:=(a,b,c+(j-1)k)$ (where the $3$rd component is understood modulo $l$). 
We have that the residue of $\gamma_j$ is $i_1-(ke/l)+e\mathbb{Z}$ for $\ulambda[0]$ and the multicharge ${\bf s}$.

 Now assume that $\eta=(a',b',c')$ is an addable  or a removable  $i_1-(ke/l) +e\mathbb{Z}$-node for $\ulambda$, different from $\gamma$. As above,  
 there exists  $j'\in \{1,\dots,l\}$ such that $\eta_{j'}:=(a',b',j'-1)$ is an addable or removable  $i_1-ke/l +e \mathbb{Z}$-node for $\ulambda[0]$  
   (and removable if and only if $\eta$ is).

   In addition,  by our definition of $\prec_{.}$, 
we have $\eta \prec_{{\bf v}} \gamma $ if and only is $\eta_{j'} \prec_{{\bf s}} \gamma_j $. 
Reciprocally, all the $i_1-ke/l +e \mathbb{Z}$-nodes are obtained in this way. 

This discussion implies that $\gamma_j$ is a good addable $i_1-ke/l$-node for $\ulambda[0]$ because $\gamma$ is a good one for $\ulambda$. We denote by $\ulambda[1]$ the $l$-partition obtained from this one by adding $\gamma_j$ to $\ulambda[0]$.  We thus have  $\ulambda[1]=\widetilde{f}^{e,{\bf s}}_{{i}_1+e-ke/l}  \ulambda' \neq 0$.

Let us now consider $\gamma_2:=(a,b,(c-1+(j-2)k))$ (where the $3$rd component is understood modulo $e$). It is an addable $i_1+(l-2)e/l+e \mathbb{Z}$-node   for  $\ulambda[1]$ and by exactly the same argument as above, we see that this is a good addable node.  Let $\ulambda[2]$ be the $l$-partition  obtained by  adding this node to $\ulambda[1]$. We obtain $ \widetilde{f}^{e,{\bf s}}_{i_1 +e-2ke/l} \widetilde{f}^{e,{\bf s}}_{i_1 +e-ke/l}  \ulambda' =\ulambda[2]$.
Continuing in this way we deduce 
$ \widetilde{f}^{e,{\bf s}}_{i_1} \ldots \widetilde{f}^{e,{\bf s}}_{i_1 +(l-2)e/l} \widetilde{f}^{e,{\bf s}}_{i_1 +(l-1)e/l}  \ulambda' =\ulambda$, as required.

\end{proof}

One can now give a general proof of the proposition. Assume that  ${\bf s}\in \mathbb{Z}^l$ is in the orbit of $(0,e/l,\ldots,(l-1)e/l)$, that $\ulambda \in \Phi_{e, {\bf s}} (n)$ and that we have  a sequence  $(i_1,\ldots,i_n)\in \mathbb{Z}^n$ such that 
 $$\widetilde{f}^{ke/l, {\bf v}}_{i_1} \ldots \widetilde{f}^{ke/l, {\bf v}}_{i_n} \emptyset =\ulambda.$$
Then, by the above lemma, if we set  ${\bf s}'=(0,e/l,\ldots,(l-1)e/l)$, we have 
$$\underbrace{\widetilde{f}^{e,{\bf s}'}_{{i}_1}\widetilde{f}^{e,{\bf s}'}_{{i}_1+ke/l}  \ldots \widetilde{f}^{e,{\bf s}'}_{{i}_1+e-ke/l}}_{l/k}  \ldots  \underbrace{\widetilde{f}^{e,{\bf s}'}_{{i}_n}\widetilde{f}^{e,{\bf s}'}_{{i}_n+ke/l}  \ldots \widetilde{f}^{e,{\bf s}'}_{{i}_n+e-ke/l}}_{l/k} \emptyset =(\lambda^1,\ldots,\lambda^{k},\lambda^1,\ldots,\lambda^{k},\ldots, \lambda^1,\ldots,\lambda^{k})$$
and thus one can conclude that 
$$\underbrace{\widetilde{f}^{e,{\bf s}}_{{i}_1}\widetilde{f}^{e,{\bf s}}_{{i}_1+ke/l}  \ldots \widetilde{f}^{e,{\bf s}}_{{i}_1+e-ke/l}}_{l/k}  \ldots  \underbrace{\widetilde{f}^{e,{\bf s}}_{{i}_n}\widetilde{f}^{e,{\bf s}}_{{i}_n+ke/l}  \ldots \widetilde{f}^{e,{\bf s}}_{{i}_n+e-ke/l}}_{l/k} \emptyset =\Psi^{e}_{{\bf s}'\to {\bf s}}(\lambda^1,\ldots,\lambda^{k},\lambda^1,\ldots,\lambda^{k},\ldots, \lambda^1,\ldots,\lambda^{k}).$$
which proves the theorem and also gives en explict way to compute the $l$-partition involved.

\begin{exa}
We take $l=2$ and $e=4$. We set $k=1$. Let $\lambda=(4.3.1)$, this is a $e/2$-regular partition and thus in $\Phi_{2,(0)} (8)$.  By Lemma \ref{H1}, we have
$\iota_1^{(0,2)}(4.3.1)=(4.3.1,4.3.1)$. Then we obtain for example $\iota_1^{(0,22)}(4.3.1)=\Psi^4_{(0,2)\to (0,22)} (4.3.1,4.3.1)=(3.2.1,4.3.2.1)$. 

\end{exa}

  \section{Hecke algebras of type $G(p,p,n)$}
In this part, we apply the results above to recover  and generalize some of the results of \cite{Hugp} and \cite{Hl}  and give precisions on them.  We will freely use the results in \cite{GwJ}. 
\subsection{Definition}

Let $\eta\in \mathbb{C}^{\times}$. Assume that  $n>2$. 
Let ${\bf s}=(s_1,\ldots,s_l)\in \mathbb{Z}^l$ and let $\eta\in \mathbb{C}^{\times}$.  The cyclotomic Hecke algebra   $\mathcal{H}_n({\bf s})$ of type $G(l,1,n)$ (also known as Ariki-Koike algebra)  is the  $\mathbb{C}$-algebra with a presentation by:
\begin{itemize}
\item generators : $T_0$, $T_1$, ..., $T_{n-1}$,
\item relations :
$$\begin{array}{rcl}
(T_{0}-\eta^{s_1})\ldots (T_{0}-\eta^{s_l})&=&0\\
(T_{i}-\eta)(T_{i}+1)&=&0 \ (1\leq i\leq n-1) \\
(T_{0}T_{1})^{2}&=&(T_{1}T_{0})^{2}\\
T_{i}T_{i+1}T_{i}&=&T_{i+1}T_{i}T_{i+1}\ (1\leq i<n)\\
T_{i}T_{j}&=&T_{j}T_{i}\ ( j\geq i+2).
\end{array}$$
 
 \end{itemize}
The cyclotomic Hecke algebra   $\mathcal{H}_n'$ of type $G(l,l,n)$ is the  $\mathbb{C}$-algebra with a presentation by :
\begin{itemize}
\item generators : $T_0$, $T_1$, ..., $T_{n-1}$,
\item relations :
\begin{align*}
 &   (T_i-{\eta})(T_i+1)=0& & \textrm{for}\    0\leq{i}\leq{n-1},&\\
 & T_iT_{i+1}T_i=T_{i+1}T_iT_{i+1}& &\textrm{for}\ 1\leq{i}\leq{n-2}, &\\
&  T_0T_{2}T_0=T_{2}T_0T_{2},& & &\\
 &   (T_1T_{0}T_2)^2=(T_{2}T_1T_{0})^2,& & &  \\
 & T_0T_j=T_jT_0& &\textrm{for}\ j>2, &\\
&  T_iT_j=T_jT_i& &\textrm{for}\ i>0\ \textrm{et}\ j>i+1, &\\
  & \underbrace{T_0T_1T_0T_1...}_{l\text{ terms}}=\underbrace{T_1T_0T_1T_0...}_{l\text{ terms}}.&& & \\
 \end{align*}
 
 \end{itemize}
From now, we assume that $\eta$ is a primitive root of order $e>1$. 
Let ${\bf s}\in \mathbb{Z}^l$ be in the orbit of $(0,e/l,\ldots,(l-1)e/l)$ then 
 the subalgebra of $\mathcal{H}_n :=\mathcal{H}_n (\bf s)$ generated by $\{T_0:=\widetilde{T_0}^{-1}T_1\widetilde{T_0},T_1,...,T_{n-1}\}$ 
is isomorphic to $\mathcal{H}_n'$. Moreover $\mathcal{H}_n$  is   $\mathbb{Z}/l\mathbb{Z}$-graded with respect to   $\mathcal{H}_n'$    with gradation 
$$\mathcal{H}_n=\bigoplus_{j=0}^{l-1}{\widetilde{T_0}^j \mathcal{H}'_{n}}.$$
As a consequence, one may use Clifford Theory to obtain results for the representation theory of  $\mathcal{H}_n'$  from the one of  $\mathcal{H}_n$.  
 To do this, we first need to recall some known results on the representation theory of $\mathcal{H}_n$

\subsection{Simple $\mathcal{H}_n$-modules}

\begin{abs} The classification of the simple $\mathcal{H}_n$-modules that we need comes from the theory of basic sets. A complete review of this can be found in \cite{GJ}
 but we quickly recall what we need here. One can define a certain set of finite dimensional 
 $\mathcal{H}_n$-modules which are parametrized by the set of $l$-partitions, they are called Specht modules
 $$\{ S^{\ulambda}\ |\ \ulambda \in \Pi^l (n)\}.$$
 These modules are non simple  (nor semisimple) in general but we have  associated composition series. Let us denote by 
 $[S^{\ulambda} : M]$ the multiplicity of $M\in \operatorname{Irr} (\mathcal{H}_n)$ in a composition series for $S^{\ulambda}$  (this is well-defined by the Jordan-H\"older theorem). Then the matrix defined by:
$$\mathcal{D} :=( [S^{\ulambda} : M])_{\ulambda \in \Pi^l( n),M\in \operatorname{Irr} (\mathcal{H}_n)}$$
controls a part of the representation theory of $\mathcal{H}_n$.This is called the {\it decomposition matrix}. 
\end{abs}
\begin{abs}  We here follow \cite[Ch.5, Ch.6]{GJ}.   
  Then one can define a pre-order $\ll_{\bf s}$ on the set of  
 $l$-partitions which depends on the choice of ${\bf s}$. We don't give the definition of this pre-order here, all we need to know is the following theorem (see 
 \cite[\S 6.7]{GJ}). 
 \end{abs}
  \begin{Th}\label{basic} Under the above hypotheses,   for all $M\in \operatorname{Irr} ( \mathcal{H}_n)$, 
  \begin{enumerate}
  \item there exists $\ulambda_M\in \Phi_{{\bf s},e} (n)$  such that $[ S^{\ulambda_M} : M]=1$,
  \item for all $\mu\vdash_l n$, if $[ S^{\umu} : M]\neq 0$ then $\umu\ll_{\bf m} \ulambda_M$.
  \end{enumerate}
  The map $M\mapsto \ulambda_M$ is injective. As a consequence, if for all $M\in \operatorname{Irr} ( \mathcal{H}_n)$ we denote  $D_{{\bf s},e}^{\ulambda_M}:=M$, we have:
$$\operatorname{Irr} (\mathcal{H}_n)=\{ D_{{\bf s},e}^\umu \ |\ \umu \in \Phi_{{\bf s},e} (n)  \}.$$
  \end{Th}

It is thus important to note that this theorem does not give one way to label the simple modules of the algebra $\mathcal{H}_n$ but in fact several ones: one for each choice of  an 
 element in the orbit of ${\bf s}$ modulo $\widehat{\mathfrak{S}}_l$. It is now natural to ask how all these parametrization are connected.  It turns out that 
  the crystal isomorphisms make the links between them. 
 
 \begin{Prop}[\cite{J}]\label{iso}
 Let ${\bf s}\in \mathbb{Z}^l$ and ${\bf s}'\in \mathbb{Z}^l$ be two multicharge in the same orbit
  then for all $\ulambda \in  \Phi_{{\bf s},e} (n)$, we have $D_{{\bf s},e}^\ulambda= D_{{\bf s}',e}^{\Psi^{e}_{{\bf s}\to {\bf s}'} (\ulambda) }$. 
 \end{Prop}

\subsection{Restriction of simple $\mathcal{H}_n$-modules}
We here set  ${\bf s}=(0,e/l,\ldots,(l-1)e/l)$. There is a natural action of the cyclic group $\mathbb{Z}/l\mathbb{Z}$ on $\Pi^l (n)$ generated by the following map:
$$(\lambda^{1},\lambda^{2},...,\lambda^{l})\mapsto{(\lambda^{l},\lambda^{1},...,\lambda^{l-2})}.$$
For $\ulambda \in \Pi^l (n)$ we denote by  $\widetilde{{\ulambda}}$ the associated equivalence class. Let 
$$r:=r(\ulambda)=\frac{l}{\textrm{Cardinality of }\ \widetilde{{\ulambda}}}.$$
The following theorem is proved in \cite{GwJ}. 

\begin{Th}
Let $\ulambda \in  \Phi_{{\bf s},e} (n)$ then we have that $\operatorname{Res} (D_{{\bf s},e}^\ulambda)$ is a direct sum of 
 $r(\ulambda)$ simple $\mathcal{H}_n'$-modules. 
\end{Th}
It is also possible to show that if $\ulambda$ and $\umu$ are in the same equivalence class then 
 $\operatorname{Res} (D_{{\bf s},e}^\ulambda)$ and  $\operatorname{Res} (D_{{\bf s},e}^\umu)$ are isomorphic. In addition,
  the simple modules appearing in the restriction of the $D_{{\bf s},e}^\ulambda$'s  determined the equivalence class of $\ulambda$.  
 As a consequence, one can obtain 
  a classification of the simple $\mathcal{H}_n$-modules knowing the numbers  $r(\ulambda)$.  Applying Proposition \ref{iso} yields:

\begin{Prop}
Let ${\bf s}'\in \mathbb{Z}^l$ in the same class as ${\bf s}=(0,e/l,\ldots,(l-1)e/l)$. 
Let $\ulambda \in  \Phi_{{\bf s}',e} (n)$ then we have that $\operatorname{Res} (D_{{\bf s}',e}^\ulambda)$ is a direct sum of 
 $r( \Psi^{e}_{{\bf s}\to ,{\bf s}} (\ulambda))$ simple $\mathcal{H}_n'$-modules. 
\end{Prop}

The above proposition gives thus an explicit way to find the number of simple modules 
 in the restriction of the simple $\mathcal{H}_n$-modules without refering to the notion of  crystal and 
  for all the known parametrization of the simples. This thus includes the usual  parametrization by set the Kleshchev $l$-partitions 
    using our isomorphisms. 
 \begin{Prop} Let ${\bf s}'\in \mathbb{Z}^l$ in the same class as ${\bf s}=(0,e/l,\ldots,(l-1)e/l)$.
 Let $\ulambda \in  \Phi_{{\bf s}',e} (n)$ then $\operatorname{Res} (D_{{\bf s}',e}^\ulambda)$ 
  splits into a sum of $x$ simple modules if and only if $\ulambda \in \operatorname{Im} (\iota^{{\bf s}'}_{l/x})$ and 
     $\ulambda \notin \operatorname{Im} (\iota^{{\bf s}'}_{l/s})$ for $s>x$.
 
 \end{Prop}
 \begin{proof}
Take first ${\bf s}'={\bf s}$.  Then by Lemma \ref{H1}, we have that 
$\ulambda \in \operatorname{Im} (\iota^{{\bf s}}_x)$  if and only if  $\ulambda$ is of the form 
$(\lambda^1,\ldots,\lambda^{k},\lambda^1,\ldots,\lambda^{k},\ldots, \lambda^1,\ldots,\lambda^{k})$ where $k=l/x$ and the result follows from the last proposition. 

Now, if ${\bf s}'\in \mathbb{Z}^l$ in the same class as ${\bf s}=(0,e/l,\ldots,(l-1)e/l)$, then  for all $\ulambda \in  \Phi_{{\bf s},e} (n)$, we have $D_{{\bf s},e}^\ulambda= D_{{\bf s}',e}^{\Psi^{e}_{{\bf s}\to {\bf s}'} (\ulambda) }$ and one can conclude noticing that 
$$\ulambda \in \operatorname{Im} (\iota^{{\bf s}}_{l/x})\iff \Psi^e_{{\bf s}'\to {\bf s}} (\ulambda) \in \operatorname{Im} (\iota^{{\bf s}'}_{l/x}).$$
by the definition of the maps and the properties of crystal isomorphisms.

 \end{proof}
 \begin{exa}
 Take $l=4$ and $e=4$. Let ${\bf s}=(0,1,2,3)$. Then we have $\ulambda:=(3.1,2,3.1,2)\in \Phi_{e,{\bf s}} (16)$. 
  Then we have that $r(\ulambda)=2$, this implies that $\operatorname{Res} (D_{{\bf s},e}^\ulambda)$  splits 
    in two simple $\mathcal{H}_n '$-modules. Note that we have $\ulambda=\iota_{2}^{\bf s} (3.1,2)$.  
 Set ${\bf s}=(0,13,26,39)$, this multicharge is very dominant  and in the same orbit as ${\bf s}$. One can compute 
  $\Psi^e_{{\bf s} \to {\bf s}'} (\ulambda)$ and we obtain $\ulambda'=(2.1,1,3.2,2.1)$. So we have that 
  $\operatorname{Res} (D_{{\bf s}',e}^{\ulambda'})$  splits 
    in two simple $\mathcal{H}_n '$-modules.
 \end{exa}

\subsection{The case $l=2$}
We assume in this part that $l=2$ and that $e$ is odd. Then one can apply the results above. In particular, for all ${\bf s}$ in the orbit 
 of $(0,e/2)$ modulo $\widehat{\mathfrak{S}}_2$ and $\ulambda \in \Phi_{e,{\bf s}} (n)$, the $\mathcal{H}_n'$-module $\operatorname{Res} (D_{{\bf s},e}^\ulambda)$ 
  splits into one or two simple modules. 

The aim is to study the 
 set of Uglov bipartitions  $\ulambda \in \Phi_{e,{\bf s}} (n)$ such that $\operatorname{Res} (D_{{\bf s},e}^\ulambda)$  splits into a sum of two simple modules.
   Such bipartitions will be called  {\it divided bipartitions} for the multicharge ${\bf s}$. This notion strongly depends on ${\bf s}$. In the case where 
    ${\bf s}=(0,e/2)$, by the results above,  these bipartitions correspond exactly  to the bipartitions of the form $(\lambda,\lambda)$ in $\Phi_{e,{\bf s}} (n)$. 
     This is exactly the set of bipartitions  $(\lambda,\lambda)$ where $\lambda$ is an $e/2$-regular partition of $n/2$ by \S \ref{flotw}.
 
%
%
%

\begin{Prop} 
 Let $N \in \mathbb{Z}_{\geq 0}$. 
We have that $(\lambda^1,\lambda^2)$ is a divided bipartition for ${\bf s}=(0,e/2+Ne)$ if and only if  we have 
$$\Psi^e_{(0,e/2+Ne)\to (Ne,e/2)} (\lambda^1,\lambda^2)= (\lambda^2,\lambda^1)$$
\end{Prop}
\begin{proof}
 Let ${\bf s}=(0,e/2+Ne)$. 
Assume that    $(\lambda^1,\lambda^2)$ is a divided bipartition. 
Let $(i_1,\ldots,i_n)\in (\mathbb{Z}/e\mathbb{Z})^n$ be such that 
 $$\widetilde{f}^{e, {\bf s}}_{i_1} \ldots \widetilde{f}^{e, {\bf s}}_{i_n} \emptyset =\ulambda$$
 Then we also have 
 $$\widetilde{f}^{e, {\bf s}}_{i_1+e/2} \ldots \widetilde{f}^{e, {\bf s}}_{i_n+e/2} \emptyset =\ulambda$$ 
 Now we have that ${\bf s}':=\tau. {\bf s}=(e/2+Ne,e)$ and by \S \ref{crt}, we obtain:
  $$\widetilde{f}^{e, {\bf s}'}_{i_1} \ldots \widetilde{f}^{e, {\bf s}'}_{i_n} \emptyset =(\lambda^2,\lambda^1)$$
 Then we also have 
 $$\widetilde{f}^{e, {\bf s}'}_{i_1+e/2} \ldots \widetilde{f}^{e, {\bf s}'}_{i_n+e/2} \emptyset =(\lambda^2,\lambda^1)$$ 
 But now  note that ${\bf s}=(Ne,e/2)={\bf s}'-(e/2,e/2)$ so it is clear that we obtain :
  $$\widetilde{f}^{e, (Ne,e/2)}_{i_1-e/2} \ldots \widetilde{f}^{e, (Ne,e/2)}_{i_n-e/2} \emptyset =(\lambda^2,\lambda^1)$$
  and this implies that $\Psi^e_{(0,e/2+Ne)\to (Ne,e/2)} (\lambda^1,\lambda^2)= (\lambda^2,\lambda^1)$. 
   Reciprocally, assume that 
   $$\Psi^e_{(0,e/2+Ne)\to (Ne,e/2)} (\lambda^1,\lambda^2)= (\lambda^2,\lambda^1).$$
   Assume that $(i_1,\ldots,i_n)\in (\mathbb{Z}/e\mathbb{Z})^n$ is such that 
 $$\widetilde{f}^{e, {\bf s}}_{i_1} \ldots \widetilde{f}^{e, {\bf s}}_{i_n} \emptyset =(\lambda^1,\lambda^2).$$
 Then we have:
  $$\widetilde{f}^{e, (Ne,e/2)}_{i_1} \ldots \widetilde{f}^{e, (Ne,e/2)}_{i_n} \emptyset =(\lambda^2,\lambda^1).$$
 We have $\tau. (Ne,e/2)=(e/2,Ne+e)$,
   and thus 
     $$\widetilde{f}^{e,(e/2,Ne+e)}_{i_1} \ldots \widetilde{f}^{e, (e/2,Ne+e)}_{i_n} \emptyset =(\lambda^1,\lambda^2).$$
As $(0,e/2+Ne)=(e/2,Ne+e)-(e/2,e/2)$, we obtain 
      $$\widetilde{f}^{e,(0,e/2+Ne)}_{i_1+e/2} \ldots \widetilde{f}^{e, (0,e/2+Ne)}_{i_n+e/2} \emptyset =(\lambda^1,\lambda^2)$$ 
     which implies that $(\lambda^1,\lambda^2)$ is a divided partition (for $(0,e/2+Ne)$). 
%

\end{proof}

\vspace{0.5cm}

\noindent {\bf Address:} \\

\noindent \textsc{Nicolas Jacon}, Universit\'e de Reims Champagne-Ardenne, UFR Sciences exactes et naturelles, Laboratoire de Math\'ematiques UMR9008
Moulin de la Housse BP 1039, 51100 Reims, FRANCE\\  \emph{nicolas.jacon@univ-reims.fr}\\

\end{document}